\documentclass[11pt]{amsproc}

\usepackage[utf8]{inputenc}
\usepackage[T1]{fontenc}
\usepackage{amsthm,amsmath,amssymb,mathtools}
\usepackage[margin=1.2in]{geometry}
\usepackage{microtype}
\usepackage{graphicx}
\usepackage[colorlinks=true,pdfstartview=FitV,linkcolor=blue,citecolor=blue,urlcolor=blue]{hyperref}

\newcommand\R{\mathbb{R}}
\newcommand\Rot{\operatorname{Rot}}
\numberwithin{equation}{section}
\mathtoolsset{showonlyrefs}

\newcommand{\vv}{V} 

\newtheorem{Thm}{Theorem}
\newtheorem{Lemma}[Thm]{Lemma}
\newtheorem{Prop}[Thm]{Proposition}
\newtheorem{Cor}[Thm]{Corollary}
\theoremstyle{remark}

\title[The symmetric octagon case]{Rigidity in the planar Ulam floating body problem with perimetral densities~$\sigma=\tfrac18,\tfrac38$ under central symmetry}

\author{Oleg Asipchuk}
\address{Oleg Asipchuk, University of Cincinnati, Department of Mathematical Sciences, Cincinnati, OH 45221, USA}
\email{asipchah@ucmail.uc.edu}

\author{Maksim Kosmakov}
\address{Maksim Kosmakov, University of Cincinnati, Department of Mathematical Sciences, Cincinnati, OH 45221, USA}
\email{kosmakmm@ucmail.uc.edu}

\author{Pavel Zatitskii}
\address{Pavel Zatitskii, University of Cincinnati, Department of Mathematical Sciences, Cincinnati, OH 45221, USA}
\email{zatitspl@ucmail.uc.edu}


\begin{document}

\begin{abstract}
We prove that the only planar, centrally symmetric, strictly convex body
$K\subset\R^2$ with $C^1$ boundary that floats in equilibrium in every
orientation for the perimetral densities $\sigma=\tfrac18$ or
$\sigma=\tfrac38$ is a disk.
\end{abstract}

\maketitle

\section*{Introduction}

This paper continues our study of Ulam's planar floating-body problem: ``Is a solid of uniform density which will float in water in every position a sphere?'', see Problem 19 in The Scottish Book \cite{SBook2015}. 
We briefly recall the geometric setting. 
In the planar case, a planar domain
$K$ of uniform density floats in equilibrium in every orientation if each
waterline cuts the boundary $\partial K$ into two arcs whose lengths are in the
fixed ratio $\sigma:(1-\sigma)$ (see Auerbach~\cite{Aue1938}); the number $\sigma$ is called the
\emph{perimetral density}. For $\sigma=\frac12$, noncircular examples were first constructed by Zindler~\cite{Zindler1921} and later placed by Auerbach in the floating-body setting. Wegner~\cite{Wegner2003,Weg2007} also constructed noncircular floating bodies for certain implicitly determined values of the physical density, which is a different parameter from the perimetral density used here. 

Bracho, Montejano, and Oliveros~\cite{BMO2001,BMO2004} proved that the disk is the only solution for the low-denominator perimetral densities $ \sigma=\frac13,  \frac14, \frac15, \frac25 $ using the  so-called Zindler carousel formalism. The cases $ \sigma=\frac13$ and $\sigma=\frac14$ were mentioned in~\cite{Aue1938} and attributed to E.~Salkowski~\cite{Salkowski1934}. The case $\sigma=\frac16$ was treated recently in~\cite{AKZsigma16}. Here we prove the corresponding result for centrally symmetric bodies when $\sigma=\frac18$ or $\sigma=\frac38$. For further background on the floating-body problem, Zindler curves, carousels, and the known examples and rigidity results, see \cite{ARSY2026,AKZsigma16,Aue1938,BMO2001,BMO2004,Wegner2003,Weg2007,Weg2019} and the references therein.

The restriction to centrally symmetric bodies is a natural setting in which
rigidity phenomena are expected. At relative density $\frac12$, Falconer
proved the corresponding rigidity statement for origin-symmetric bodies in
dimension three~\cite{Falconer1983}, while the general higher-dimensional
statement follows from an earlier theorem of
Schneider~\cite{Schneider1970}.
A short modern proof was subsequently given
by Florentin, Sch\"utt, Werner, and Zhang
\cite{FlorentinSchuttWernerZhang2022}. This rigidity contrasts with Ryabogin's negative answer to the higher-dimensional version of Ulam's problem: without the symmetry
assumption, there exist strictly convex nonspherical bodies of relative
density $\frac12$ which float in equilibrium in every orientation
\cite{Ryabogin2022,Ryabogin2023}.

The main result is the following.

\begin{Thm}\label{thm:main}
Let $K\subset\R^2$ be a centrally symmetric strictly convex planar body with
$C^1$ boundary. If $K$ floats in equilibrium in every orientation with
perimetral density $\sigma=\frac18$ or $\sigma=\frac38$, then $K$ is a disk.
\end{Thm}

The starting point of the proof is a system of differential equations derived
by Bracho, Montejano, and Oliveros~\cite{BMO2001,BMO2004} for the angles
associated with the floating configuration. Central symmetry reduces this
system to two variables. After a change of variables, the equations for
$\sigma=\frac18$ and $\sigma=\frac38$ have the same Hamiltonian form, with
the two cases distinguished only by the sign of a parameter. We study the
periodic solutions of this system and estimate their periods in terms of the
Hamiltonian level. For a noncircular body, these period estimates must also
satisfy a discrete rotational constraint and geometric bounds involving the
perimeter. The resulting conditions are incompatible, ruling out all
nonconstant solutions. The remaining constant solution corresponds to a
disk.

\medskip
\noindent
\textit{Acknowledgments}. We would like to thank Dmitry Ryabogin for drawing our attention to this problem and providing valuable feedback.

\section{The carousel construction}
\label{sec:known-facts}

We recall the carousel construction; see \cite{BMO2004}. Let $K$ be a strictly
convex planar body with perimeter $P$, and let
$\gamma:\R/P\mathbb Z\to\partial K$ be an arc-length parametrization. For
$\sigma\in\{1/8,3/8\}$ set $h=\sigma P$ and
$\beta_i(t)=\gamma(t+(i-1)h)$, $i=1,\ldots,8$. Throughout the paper $K$ is
centrally symmetric, with center at the origin, therefore we have
$\gamma(t+P/2)=-\gamma(t)$.

By Auerbach's formulation of the planar floating condition \cite{Aue1938}, the
floating chord has constant length. Since the problem is scale invariant, we may
normalize this length to be $2$:
\begin{equation}\label{eq:carousel-sides}
    |\beta_{i+1}(t)-\beta_i(t)|=2,
    \qquad i=1,\ldots,8,
\end{equation}
where indices are understood modulo $8$. We consider the moving equilateral octagon $\vv =\vv(t)$ with the vertices $\beta_i(t), i=1,\dots,8$. It is a convex octagon for $\sigma=1/8$, see Figure~\ref{fig:octagon-1-8}. For $\sigma=3/8$ it is an index-three, or star-order, octagon, see Figure~\ref{fig:octagon-3-8}.

Let $x_i$ denote the angle of the octagon $\vv$ at $\beta_i$, that is, the angle between the sides $\beta_{i+1}-\beta_i$ and $\beta_i-\beta_{i-1}$. Let $\alpha_i$ denote the angle between the side $\beta_{i+1}-\beta_i$ and the tangent $\dot\beta_i$
to~$\partial K$ at~$\beta_i$,  see Figures~\ref{fig:octagon-1-8} and  \ref{fig:octagon-3-8}. The constant side-length condition~\eqref{eq:carousel-sides} implies that the angle between the same side $\beta_{i+1}-\beta_i$ and the tangent $\dot\beta_{i+1}$ at $\beta_{i+1}$ is the same.
By central symmetry,
\begin{equation}\label{eq:oct-central}
    \beta_{i+4}=-\beta_i,
    \qquad x_{i+4}=x_i,
    \qquad \alpha_{i+4}=\alpha_i.
\end{equation}
With this angle convention,
\begin{equation}\label{eq:x-alpha}
    x_i+\alpha_{i-1}+\alpha_i=\pi,\qquad i=1,\ldots,4,
\end{equation}
where indices are taken modulo $4$.

\begin{figure}[h]
    \centering
    \includegraphics[width=0.75\linewidth]{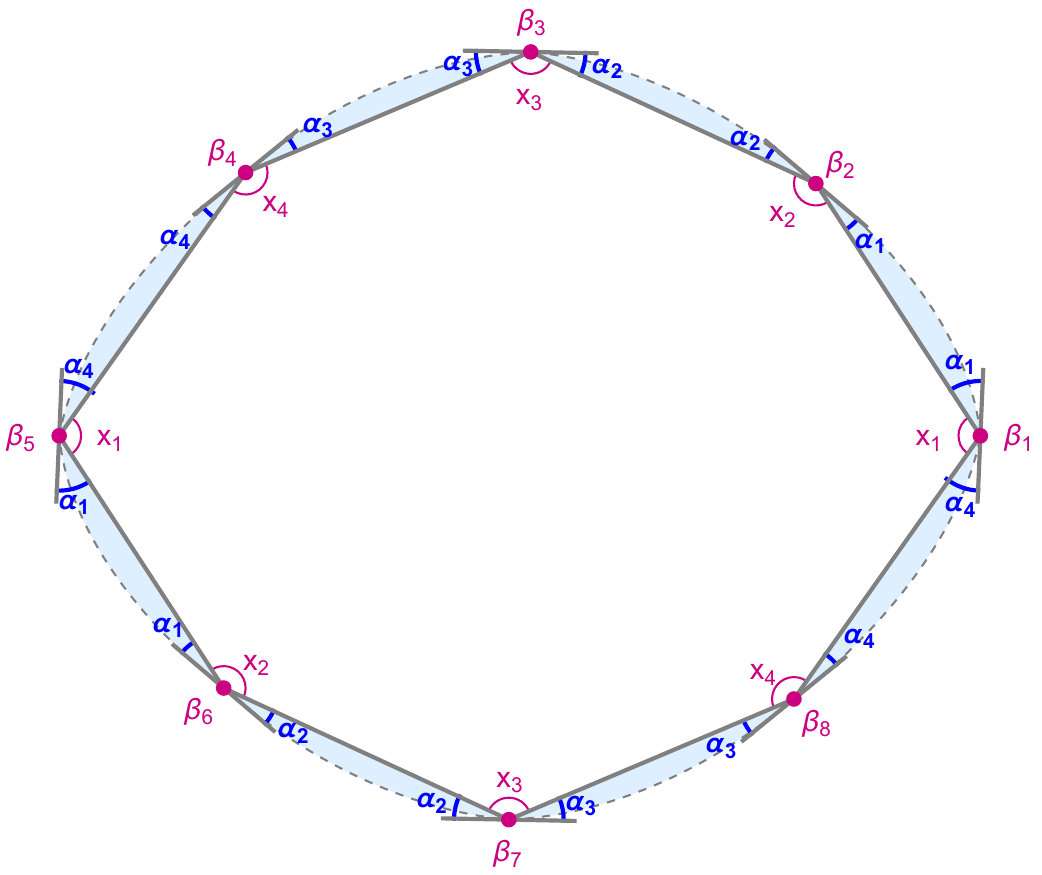}
    \caption{A centrally symmetric convex body $K$ with an inscribed equilateral octagon~$\vv(t)$ corresponding to $\sigma=\frac18$.}
    \label{fig:octagon-1-8}
\end{figure}

\begin{figure}[h]
    \centering
    \includegraphics[width=0.75\linewidth]{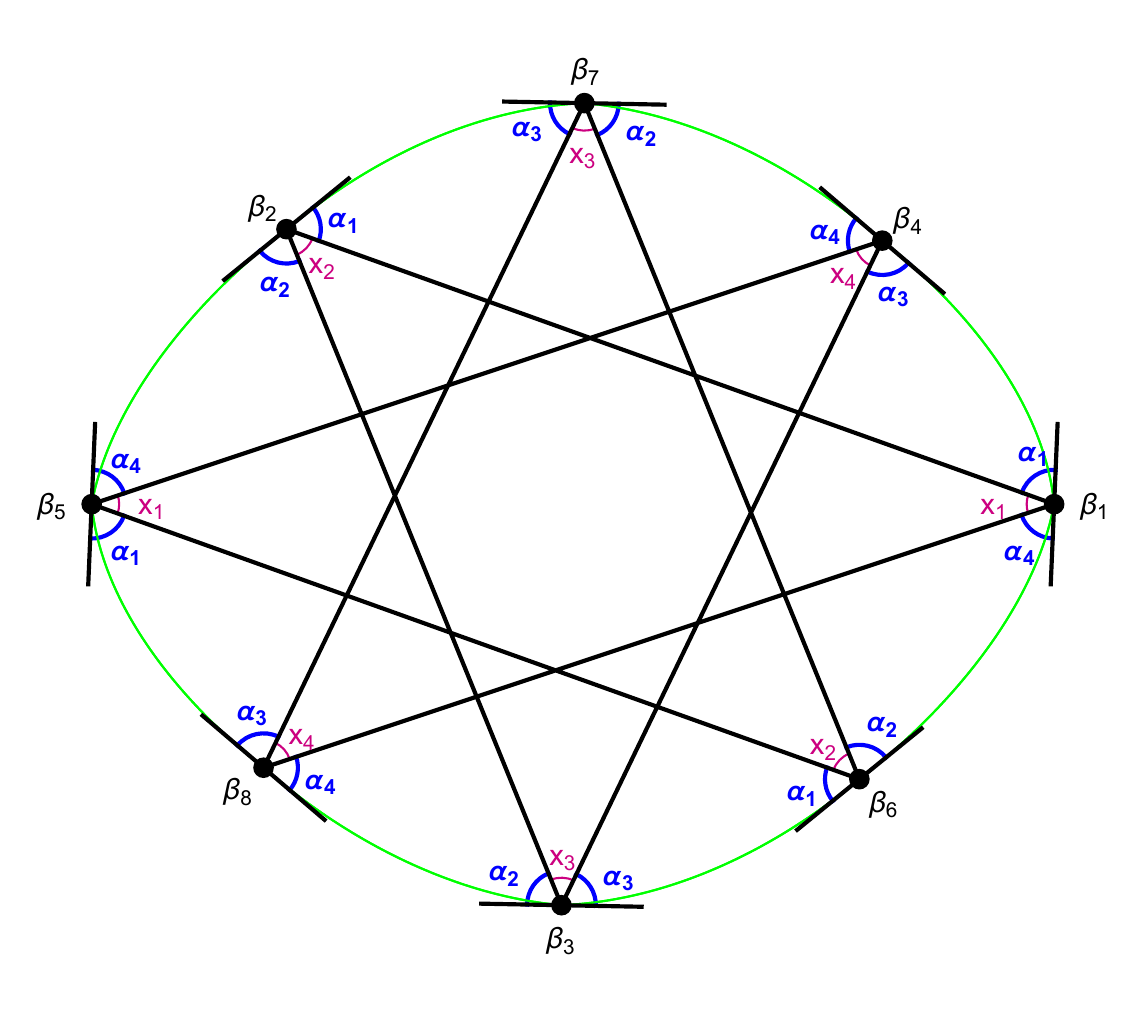}
    \caption{A centrally symmetric convex body $K$ with an inscribed index-three octagon ~$\vv(t)$ corresponding to $\sigma=3/8$.}
    \label{fig:octagon-3-8}
\end{figure}

We use the following theorem of Bracho, Montejano, and Oliveros, adapted to our cases.

\begin{Thm}[{\cite{BMO2001,BMO2004}}]\label{thm:BMO-carousel}
The angles $x_i,\alpha_i$ satisfy
\begin{equation} 
    \dot x_i=\sin\alpha_{i-1}-\sin\alpha_i , \qquad i=1,\dots,4.\label{eq:BMO}
\end{equation}
\end{Thm}

Central symmetry gives the following identities and bounds for these angles.

\begin{Lemma}\label{lem:known-angle-identities}
For $\sigma\in\{\frac18,\frac38\}$,
\begin{align}
    \alpha_1+\alpha_2+\alpha_3+\alpha_4&=4\pi\sigma,\label{eq:sum-alpha}\\
    x_1+x_3=x_2+x_4&=2\pi(1-2\sigma).\label{eq:opposite-angle-sums}
\end{align}
Moreover, $\pi/2<x_i<\pi$ when $\sigma=1/8$, while $0<x_i<\pi/2$ when
$\sigma=3/8$, for $i=1,\ldots,4$. In both cases,
\begin{equation}\label{eq:sine-balance}
    \sin\alpha_1+\sin\alpha_3=\sin\alpha_2+\sin\alpha_4 .
\end{equation}
\end{Lemma}

\begin{proof}
 The winding number of the octagon is $8\sigma$. Thus the sum of the exterior angles is  is $16\pi\sigma$, and \eqref{eq:oct-central} gives
$x_1+x_2+x_3+x_4=4\pi-8\pi\sigma$. Summing \eqref{eq:x-alpha} over $i=1,\ldots,4$, we get
\eqref{eq:sum-alpha}. Adding \eqref{eq:x-alpha} for $i=1,3$ and for $i=2,4$, we obtain
\eqref{eq:opposite-angle-sums}. The bounds for $x_i$ now follow from~\eqref{eq:opposite-angle-sums} and the trivial estimates $0<x_i<\pi$ for $i=1,\ldots,4$.

Finally, differentiating $x_1+x_3=x_2+x_4$ and using \eqref{eq:BMO}, we prove
\eqref{eq:sine-balance} in both cases.
\end{proof}

\begin{Lemma}\label{lem:BMO-equilibrium-circle}
In each of the cases $\sigma=1/8$ and $\sigma=3/8$, every equilibrium solution
of \eqref{eq:BMO} corresponds to a circle.
\end{Lemma}

\begin{proof}
At equilibrium, \eqref{eq:BMO} gives $\sin\alpha_{i-1}=\sin\alpha_i$ for all
$i$. Since $\alpha_{i-1}+\alpha_i=\pi-x_i<\pi$, we have
$\alpha_{i-1}=\alpha_i$ for every  $i$. Thus all $\alpha_i$ are equal, and then
\eqref{eq:x-alpha} implies that all $x_i$ are equal. Thus all moving octagons ~$\vv(t)$ 
are congruent. In particular, the length of
the four-step diagonal $\beta_5(t)-\beta_1(t)$ is fixed. By central symmetry,
$|\beta_5(t)-\beta_1(t)|=2|\beta_1(t)|$,
therefore $|\gamma(t)|=|\beta_1(t)|$ is constant. Thus $\partial K$ is a
circle centered at the origin, and $K$ is a disk.
\end{proof}

\begin{Lemma}\label{lem:octagonal-perimeter-upper}
In both cases $\sigma=1/8$ and $\sigma=3/8$, under the normalization that the
floating chord has length $2$, one has $P=|\partial K|\le 8\pi$.
\end{Lemma}

\begin{proof}
By central symmetry, $\beta_5=-\beta_1$. Therefore
\[
    2|\beta_1|=|\beta_5-\beta_1|
    \le \sum_{j=1}^4 |\beta_{j+1}-\beta_j|=8.
\]
Thus $K$ is contained in the disk of radius $4$ centered at the origin. Since
$K$ is convex, monotonicity of perimeter gives $P\le 8\pi$.
\end{proof}

\section{Hamiltonian system}
\label{sec:2D-system}

For $\sigma\in\{\frac18,\frac38\}$, set
\begin{equation}\label{eq:theta-mu}
    \theta=\pi-2\pi\sigma
    \qquad \text{ and } \qquad 
    \mu=\cos(2\pi\sigma).
\end{equation}
By Lemma~\ref{lem:known-angle-identities}, $x_1+x_3=x_2+x_4=2\theta$ in both cases. Define 
\[r=x_1-\theta, \qquad  w=x_2-\theta,\]
then
\[x_1=\theta+r, \quad x_3=\theta-r, \quad x_2=\theta+w, \quad x_4=\theta-w.\]
The regular carousel configuration corresponds to $r=w=0$, and
by Lemma~\ref{lem:known-angle-identities} the admissible domain is
\begin{equation}\label{eq:rw-domain}
    |r|<\frac\pi4,    \qquad |w|<\frac\pi4 .
\end{equation}

\begin{Prop}\label{prop:2D-system}
After the time change
\begin{equation}\label{eq:s-time}
    \frac{ds}{dt}=\frac1{\Lambda(r,w)}, \qquad \Lambda(r,w):=\sqrt{4(\cos r+\mu)(\cos w+\mu)+2},
\end{equation}
the centrally symmetric octagon dynamics can be reduced to
\begin{equation}\label{eq:rw-system}
    \frac{dr}{ds}=2(\cos r+\mu)\sin w, \qquad \frac{dw}{ds}=-2(\cos w+\mu)\sin r.
\end{equation}
Moreover, the system \eqref{eq:rw-system} is Hamiltonian   with
\begin{equation}\label{eq:Hamiltonian-rw}
    \mathcal H_\mu(r,w)=-2(\cos r+\mu)(\cos w+\mu).
\end{equation}
That is,
\begin{equation}\label{eq:Hamiltonian-equations}
    \frac{dr}{ds}=\frac{\partial \mathcal H_\mu}{\partial w}, \qquad \frac{dw}{ds}=-\frac{\partial \mathcal H_\mu}{\partial r}.
\end{equation}
\end{Prop}

\begin{proof}
Put $\psi=\alpha_1+(r+w)/2$. Since $x_1=\theta+r$ and $x_2=\theta+w$, the
relations \eqref{eq:x-alpha} and Lemma~\ref{lem:known-angle-identities} give
\begin{equation}\label{eq:alpha-rw-unified}
    \alpha_1=\psi-\frac{r+w}{2}, \quad \alpha_3=\psi+\frac{r+w}{2}, \quad
    \alpha_2=\pi-\theta-\psi+\frac{r-w}{2}, \quad \alpha_4=\pi-\theta-\psi-\frac{r-w}{2}.
\end{equation}
Substitution into the sine-balance identity \eqref{eq:sine-balance} gives
\begin{equation}\label{eq:psi-balance}
    \sin\psi\cos\frac{r+w}{2}
    =
    \sin(\pi-\theta-\psi)\cos\frac{r-w}{2}.
\end{equation}
Set $C=\cos\frac{r+w}{2}$ and $D=\cos\frac{r-w}{2}$. By \eqref{eq:theta-mu} we have
$\cos\theta=-\mu$ and $\sin\theta=1/\sqrt2$, thus
\eqref{eq:psi-balance} becomes
\begin{equation}\label{eq:psi-linear}
    (C+\mu D)\sin\psi-\frac{D}{\sqrt2}\cos\psi=0.
\end{equation}
 
The positivity of the four angles in \eqref{eq:alpha-rw-unified} implies $0<\psi<\pi$. Moreover, on the square \eqref{eq:rw-domain}, both $D$ and $C+\mu D$ are positive, thus \eqref{eq:psi-linear} gives
$\tan\psi>0$ and  $0<\psi<\pi/2$. Therefore we have
\begin{equation}\label{eq:psi-sign-unified}
    \cos\psi=\frac{2(C+\mu D)}{\Lambda},\qquad
    \sin\psi=\frac{\sqrt2D}{\Lambda},
\end{equation}
where
\begin{equation}\label{eq:Lambda-CD}
    \Lambda^2=4\left((C+\mu D)^2+\frac{D^2}{2}\right)=4(\cos r+\mu)(\cos w+\mu)+2.
\end{equation}
The last equality uses $C^2+D^2=1+\cos r\cos w$ and
$2CD=\cos r+\cos w$.

By \eqref{eq:BMO}, $\dot r=\sin\alpha_4-\sin\alpha_1$ and
$\dot w=\sin\alpha_1-\sin\alpha_2$. Using
\eqref{eq:alpha-rw-unified} and \eqref{eq:psi-sign-unified},  we obtain
\[
    \Lambda\dot r=2(\cos r+\mu)\sin w, \qquad \Lambda\dot w=-2(\cos w+\mu)\sin r.
\]
After the time change \eqref{eq:s-time}, this proves \eqref{eq:rw-system}. Finally,  \eqref{eq:Hamiltonian-equations} follows from a direct calculation. 
\end{proof}

\begin{Cor}\label{cor:Lambda-Hamiltonian}
Along every orbit of \eqref{eq:rw-system}, both $\mathcal H_\mu$ and
$\Lambda$ are constant and 
\begin{equation}\label{eq:first-integral}
    \Lambda^2=4(\cos r+\mu)(\cos w+\mu)+2=2-2\mathcal H_\mu.
\end{equation}
\end{Cor}

\begin{proof}
The Hamiltonian is preserved by Proposition~\ref{prop:2D-system}. Equation \eqref{eq:first-integral} follows from  \eqref{eq:s-time}   and \eqref{eq:Hamiltonian-rw}.
\end{proof}

\section{One-dimensional reduction and period bounds}
\label{sec:period-bounds}

Fix a nonconstant orbit. Since $\cos r+\mu>0$ on the admissible domain, the
equation
\begin{equation}\label{eq:r-equation}
    \frac{dr}{ds}=2(\cos r+\mu)\sin w
\end{equation}
shows that the turning points of the $r$-motion occur exactly when $w=0$.
Choose one such point and write it as $(r,w)=(\lambda,0)$, where
$0<\lambda<\pi/4$, and set $L=\cos\lambda$. Since
$(\cos r+\mu)(\cos w+\mu)$ is constant on the orbit, evaluating it at
$(\lambda,0)$ gives
\begin{equation}\label{eq:first-integral-turning}
    (\cos r+\mu)(\cos w+\mu)=(L+\mu)(1+\mu)
\end{equation}
along the entire orbit, and thus
  $\sin^2 w  = 1-\left(\frac{(L+\mu)(1+\mu)}{\cos r+\mu}-\mu\right)^2$. After squaring \eqref{eq:r-equation} and substituting this expression, we get
  \[ 
  \left(\frac{dr}{ds}\right)^2 = 4\left((\cos r+\mu)^2-\left((L+\mu)(1+\mu)-\mu(\cos r+\mu)\right)^2\right). 
  \]
  Factoring the difference of squares,
  we obtain \begin{equation}\label{eq:rw-1d}
  \left(\frac{dr}{ds}\right)^2 = 4(1+\mu)(\cos r-L) \bigl((1-\mu)\cos r+(1+\mu)L+2\mu\bigr).
  \end{equation}
 Therefore the minimal $s$-period is given by
\begin{equation}\label{eq:Ts-period}
    T_s
    :=  \int_{-\lambda}^\lambda \frac{dr}{\sqrt{(1+\mu)(\cos r-L)\bigl((1-\mu)\cos r+(1+\mu)L+2\mu\bigr)}}.
\end{equation}
By Corollary~\ref{cor:Lambda-Hamiltonian}, $\Lambda$ is constant on the orbit, and we can evaluate it at
the turning point:
$\Lambda^2=4(L+\mu)(1+\mu)+2=4(1+\mu)(1+L)$ because $\mu^2=1/2$. Since the integrand in \eqref{eq:Ts-period} is even in $r$, we can express the 
minimal $t$-period as
\begin{equation}\label{eq:Tt-period}
    T_t
    :=\Lambda T_s = 4\sqrt{1+L}\int_0^\lambda \frac{dr}{\sqrt{(\cos r-L) \bigl((1-\mu)\cos r+(1+\mu)L+2\mu\bigr)}}.
\end{equation}

\begin{Prop}\label{prop:period-bounds}
 For $\sigma=1/8$, the minimal $t$-period $T_t$ of every
nonconstant admissible orbit satisfies  $2\pi<T_t<8$.
For $\sigma=3/8$, we have 
    $T_t>5\pi.$
\end{Prop}

\begin{proof}
 For $r\in(0,\lambda)$ we have $L<\cos r<1$. Since $1-\mu>0$,
\[
\frac{2(L+\mu)(1+L)}{1+\cos r}<2(L+\mu)<(1-\mu)\cos r+(1+\mu)L+2\mu<(1+\mu)(1+L)<\frac{2(1+\mu)(1+L)}{1+\cos r}.
\]

From \eqref{eq:Tt-period}, we get
\[
\frac{4}{\sqrt{2(1+\mu)}}I_\lambda<T_t<\frac{4}{\sqrt{2(L+\mu)}}I_\lambda,
\qquad
I_\lambda:=\int_0^\lambda \sqrt{\frac{1+\cos r}{\cos r-L}}\,dr.
\]
Since $I_\lambda=\pi$, we have 
\[
\frac{4\pi}{\sqrt{2(1+\mu)}}<T_t<\frac{4\pi}{\sqrt{2(L+\mu)}}.
\]
For $\sigma=1/8$, $\mu=1/\sqrt2$ and $L=\cos\lambda>1/\sqrt2$, thus $2\pi<T_t<8$. For $\sigma=3/8$, $\mu=-1/\sqrt2$, and the lower bound gives $T_t>5\pi$.
\end{proof}

\section{Proof of the main theorem}
\label{sec:proof}

Assume, for contradiction, that $K$ satisfies the hypotheses of
Theorem~\ref{thm:main} but is not a disk. Let
$\Rot(K)=\{\theta\in[0,2\pi):R_\theta(K)=K\}$ be its rotational symmetry
group. Since $K$ is not a disk, $\Rot(K)$ is a finite cyclic group. Since $K$
is centrally symmetric, its order is even, so $\Rot(K)=\langle R_{\pi/k}\rangle$
for some $k\in\mathbb N$.

Choose $\gamma$ to be positively oriented. The rotation $R_{\pi/k}$ preserves
arc length and orientation on $\partial K$. Since $R_{\pi/k}$ is the smallest
positive rotation in $\Rot(K)$,
\begin{equation}\label{eq:rotation-shift}
    R_{\pi/k}\gamma(t)=\gamma\left(t+\frac{P}{2k}\right).
\end{equation}

By Lemma~\ref{lem:BMO-equilibrium-circle}, the reduced orbit is nonconstant.
Let $T$ denote its minimal period in the original $t$-parameter. Since the
shift in \eqref{eq:rotation-shift} preserves the reduced variables, there
exists $m\in\mathbb N$ such that
\begin{equation}\label{eq:period-quantization}
    \frac{P}{2k}=mT.
\end{equation}
Thus $P=2kmT$.

Consider first $\sigma=1/8$. By Proposition~\ref{prop:period-bounds},
$T>2\pi$. Combining this with Lemma~\ref{lem:octagonal-perimeter-upper} gives
$2km\cdot2\pi<P\le8\pi$, thus $km<2$. Since $km$ is a positive integer, we get
$km=1$ and $P=2T$. On the other hand, each boundary arc from $\beta_i$ to
$\beta_{i+1}$ has length $P/8$, while the corresponding chord has length $2$.
Thus $P/8\ge2$, and so $T=P/2\ge8$. This contradicts the upper bound $T<8$ in
Proposition~\ref{prop:period-bounds}. Thus $K$ must be a disk for
$\sigma=1/8$.

Now consider $\sigma=3/8$. By Proposition~\ref{prop:period-bounds},
$T>5\pi$. Since $P=2kmT$ and $km\ge1$, we get $P>10\pi$. This contradicts
Lemma~\ref{lem:octagonal-perimeter-upper}. Thus $K$ must also be a disk for
$\sigma=3/8$.


\begin{thebibliography}{99}

\bibitem{ARSY2026}
M.~A. Alfonseca, D.~Ryabogin, A.~Stancu, and V.~Yaskin.
\newblock \emph{On flotation, stability and related questions: A survey}.
\newblock To appear in \emph{New Probes into Discrete and Convex Geometry},
J.~Pach and G.~T\'oth, eds., Springer, 2026.

\bibitem{AKZsigma16} O.~Asipchuk, M.~Kosmakov, and P.~Zatitskii. \newblock \emph{Rigidity in the planar Ulam floating body problem with perimetral density $\sigma=\frac{1}{6}$}. \newblock Preprint, \href{https://arxiv.org/abs/2604.10330}{arXiv:2604.10330}, 2026.

\bibitem{Aue1938}
H.~Auerbach.
\newblock \emph{Sur un probl\`eme de M.~Ulam concernant l'\'equilibre des
corps flottants}.
\newblock \emph{Studia Math.} \textbf{7} (1938), 121--142.

\bibitem{BMO2001}
J.~Bracho, L.~Montejano, and D.~Oliveros.
\newblock \emph{A classification theorem for Zindler carrousels}.
\newblock \emph{J. Dynam. Control Systems} \textbf{7} (2001), no.~3,
367--384.

\bibitem{BMO2004}
J.~Bracho, L.~Montejano, and D.~Oliveros.
\newblock \emph{Carousels, Zindler curves and the floating body problem}.
\newblock \emph{Period. Math. Hungar.} \textbf{49} (2004), no.~2, 9--23.

\bibitem{Falconer1983}
K.~J. Falconer.
\newblock \emph{Applications of a result on spherical integration to the
theory of convex sets}.
\newblock \emph{Amer. Math. Monthly} \textbf{90} (1983), no.~10, 690--693.

\bibitem{FlorentinSchuttWernerZhang2022}
D.~I. Florentin, C.~Sch\"utt, E.~M. Werner, and N.~Zhang.
\newblock \emph{Convex floating bodies of equilibrium}.
\newblock \emph{Proc. Amer. Math. Soc.} \textbf{150} (2022), no.~7,
3037--3048.


\bibitem{SBook2015}
R.~D. Mauldin, editor.
\newblock \emph{The Scottish Book: Mathematics from the Scottish Caf\'e,
with Selected Problems from the New Scottish Book}.
\newblock Birkh\"auser/Springer, Cham, 2nd updated and enlarged ed., 2015.

\bibitem{Ryabogin2022}
D.~Ryabogin.
\newblock \emph{A negative answer to Ulam's Problem 19 from the Scottish
Book}.
\newblock \emph{Ann. of Math. (2)} \textbf{195} (2022), no.~3, 1111--1150.

\bibitem{Ryabogin2023}
D.~Ryabogin.
\newblock \emph{On bodies floating in equilibrium in every orientation}.
\newblock \emph{Geom. Dedicata} \textbf{217} (2023), no.~4,
Paper No.~70.

\bibitem{Salkowski1934}
E.~Salkowski. 
\newblock \emph{Eine kennzeichnende Eigenschaft des Kreises}.
\newblock \emph{Sb. d. Heidelberger Akad. d. Wissensch., Math.-nat. Klasse} (1934) p. 57--62.

\bibitem{Schneider1970}
R.~Schneider.
\newblock \emph{Functional equations connected with rotations and their
geometric applications}.
\newblock \emph{Enseign. Math. (2)} \textbf{16} (1970), 297--305.

\bibitem{Wegner2003}
F.~Wegner.
\newblock \emph{Floating bodies of equilibrium}.
\newblock \emph{Stud. Appl. Math.} \textbf{111} (2003), 167--183.

\bibitem{Weg2007} F.~J. Wegner. \newblock \emph{Floating bodies of equilibrium in 2D, the tire track problem and electrons in a parabolic magnetic field}. \newblock Preprint, \href{https://arxiv.org/abs/physics/0701241}{arXiv:physics/0701241}, 2007.

\bibitem{Weg2019} F.~J. Wegner. \newblock 
\emph{From Elastica to Floating Bodies of Equilibrium}. \newblock Preprint, \href{https://arxiv.org/abs/1909.12596}{arXiv:1909.12596}, 2019.

\bibitem{Zindler1921}
K.~Zindler.
\newblock \emph{\"Uber konvexe Gebilde. II. Teil}.
\newblock \emph{Monatsh. Math. Phys.} \textbf{31} (1921), 25--56.

\end{thebibliography}
\end{document}